\documentclass[letterpaper, 10 pt, onecolumn, draftcls]{ieeeconf}  

\IEEEoverridecommandlockouts                              
\usepackage{graphicx}
\usepackage{amsmath} 
\usepackage{amssymb}  
\newtheorem{mydef}{Definition} [section]
\newtheorem{thm} [mydef]{Theorem}
\newtheorem{lemma} [mydef]{Lemma}
\newtheorem{prop} [mydef]{Proposition}

\newtheorem{remark} [mydef]{Remark}
\newtheorem{ex} [mydef]{Example}

\title{\LARGE \bf
An Augmented Observer for the Distributed Estimation Problem for LTI Systems
}

\author{Shinkyu Park and Nuno C. Martins
\thanks{This work is partially funded by NSF CPS grant No. 0931878,  ONR AppEl Center and the Multiscale Systems Center, one of six research centers funded under the Focus Center Research Program. }
\thanks{Shinkyu Park and Nuno C. Martins are with the Department of Electrical and Computer Engineering, University of Maryland College Park, College Park, MD 20742-4450, USA.
        {\tt\small \{skpark, nmartins\}@umd.edu}}%
}

\begin{document}

\maketitle
\thispagestyle{empty}
\pagestyle{empty}

\begin{abstract}
This paper studies a network of observers for a distributed estimation problem, where each observer assesses a portion of output of a given LTI system. The goal of each observer is to compute a state estimate that asymptotically converges to the state of the LTI system. We consider there is a sparsity constraint that restricts interconnections between observers.

We provide a sufficient condition for the existence of parameters for the observers which achieve the convergence of the state estimates to the state of the LTI system. In particular, this condition can be written in terms of the eigenvalues of the Laplacian matrix of the underlying communication graph and the spectral radius of the dynamic matrix of the LTI system.
\end{abstract}

\section{INTRODUCTION}
In this paper, we discuss a distributed estimation problem for a system observed by a network of $m$ LTI observers \footnote{Without loss of generality, we assume that $m > 2$.}. Consider a LTI system is given as follows with the output vector $y(k)$\footnote{In the presence of bounded noise, our result gives bounded estimation error. Since the way this can be proven is same as one presented in this paper, we omit noise terms in state space representations due to the space constraint.}.

\begin{equation} \label{state_equation}
\begin{split}
	x(k+1) = Ax(k)\\
	y(k) = C x(k)
\end{split}
\end{equation}
where 
\begin{equation*}
\begin{split}
	y(k) = \left ( y_1^T(k), \cdots, y_m^T(k) \right )^T, C = \left ( C_1^T, \cdots, C_m^T\right )^T \\
	x(k) \in \mathbb{R}^{n}, y_i(k) \in \mathbb{R}^{r_i}
\end{split} \nonumber
\end{equation*}

Each measurement $y_i(k)$ is assessed by observer $i$, and interconnected observers form a network of observers. Each observer is allowed to share its local measurement and state estimate with nearby observers (neighbors) via communication links, which is subject to communication constraints\footnote{Here we assume that the communication links are bidirectional and time-invariant. The topology of a network of observers can be represented by a graph, whose vertex set is a set of observers and edge set is a set of communication links. Henceforth, we refer to this graph as a \textit{communication graph}.}. The communication constraints introduce a sparsity pattern in the formulation of our distributed estimation problem. Henceforth, we refer to this constraint as a sparsity constraint.

Our main goal is to design a network of observers which cooperatively computes the state of the system described by \eqref{state_equation}. In other words, let $\hat{x}_i(k)$ be the state estimate by observer $i$ then our goal is to have $\lim_{k \to \infty} ||\hat{x}_i(k) - x(k)|| = 0$ for all $i \in \{1, ~ \cdots, m\}$. The main challenge in this distributed estimation problem comes from the limitations that no single observer can compute the state of the system only with its local measurement, and the exchange of information is restricted by a sparsity constraint. Thus, the classical system theory cannot be \textit{directly} applied to finding such observers.

Similar distributed estimation problem for (dynamical) linear systems has been studied with various approaches. For example, in \cite{khan2010_cdc, matei2010_cdc, delouille2006, smith2007, khan09, khan2008_tsp, olfati-saber07_ieee_cdc, carli2007_ieee_cdc, alriksson2006_mtns, bai2011_acc}, the design of a linear observer is discussed. In particular, in \cite{khan2008_tsp, olfati-saber07_ieee_cdc, carli2007_ieee_cdc, alriksson2006_mtns, bai2011_acc}, different types of consensus-based Kalman filtering are proposed, while a nonlinear approach (moving horizon estimation algorithm) \cite{farina10} is utilized in a distributed estimation problem.

Among the previous works,  Khan \textit{et. al.} \cite{khan2010_cdc} proposed a consensus-step estimator for distributed state estimation, and other notions to provide sufficient conditions for the stability of the proposed estimator. Their work focused on finding design parameters -- consensus gain and observer gain -- for the estimator, depending on a quantity called "network tracking capacity". The design procedure is simpler than ordinary observer design procedures due to the simplification in computing parameters, and distributed computation of the parameters is possible. However, from a system theoretic point of view, the proposed method works under assumptions that are stronger than ours.

Also Matei and Baras \cite{matei2010_cdc} proposed a consensus-based linear distributed estimation algorithm. Unlike the work by Khan \textit{et. al.} \cite{khan2010_cdc}, the authors assume that a consensus gain is given, and they concentrate on finding sub-optimal observer parameters in a closed form. Also sufficient conditions for the stability of the observer are presented. Since the choice of the consensus gain reflects the characteristics of an underlying communication graph, this work does not consider the effect of the underlying graph on their distributed estimation problem.

In this paper, we consider, given a sparsity constraint, the design of a network of \textit{augmented} observers for the distributed estimation problem. A key contribution of this work is to provide a sufficient condition for the existence of augmented observers, where the state estimate by each observer asymptotically converges to the state of the LTI system. First we investigate under what condition the underlying communication graph is capable of estimating the state of the LTI system in terms of eigenvalues of the underlying graph and the spectral radius of the dynamic matrix of the LTI system. Under this condition, we prove the existence of a network of augmented observers that fulfills our objective.

%

\textbf{The following notation is adopted:}
\begin{itemize}
	\item Let observer $i$ and observer $j$ be neighbors if there exists a communication link between the two observers. Then a communication graph, subject to a sparsity constraint, can be described by the Laplacian matrix:
		\begin{equation} \label{eq_network_laplacian}
		\begin{split}
			[L]_{ij} = \left\{ 
				\begin{array} {l l}
					-1 & \quad \text{if $i \neq j$}\\
					 & \quad \text{and $i$ \& $j$} \text{ are neighbors}\\ 
					-\sum_{l \neq i} [L]_{il} & \quad \text{if $i = j$} \\
					0 &  \quad \text{otherwise} \\
				\end{array} \right.
		\end{split}
		\end{equation}
		where $[L]_{ij}$ is $i,j$-th element of a matrix $L$. Notice that $L$ is a $m \times m$ symmetric matrix; hence it has $m$ real eigenvalues. An eigenvalue of $L$ and its corresponding eigenvector are denoted by $\lambda_i$ and $v_i$, respectively. Without loss of generality, we assume that $\lambda_1 \leq \lambda_2 \leq \cdots \leq \lambda_m$.
	
	\item (Pattern Operator) \cite{sabau2010_allerton} \label{def_pattern}
	$\mathcal{P}(L) \in \{0, 1\}^{m \times m}$ to be the binary matrix
	\begin{equation}
		[\mathcal{P}(L)]_{ij} \overset{def}{=} \left \{
						\begin{array} {l l}
							0 & \quad \text{if the block}~ [L]_{ij} = 0 \\
							1 & \quad \text{otherwise}
						\end{array} 
					\right.
	\end{equation}
	We define $\mathcal{P}(L_1) \leq \mathcal{P}(L_2)$ if $[\mathcal{P}(L_1)]_{ij} \leq [\mathcal{P}(L_2)]_{ij}$ for all $i,j \in \{1, ~ \cdots, ~ m\}$.

	\item $I_n$ is a $n \times n$ identity matrix, $e_i$ is $i$-th column of $I_m$, and $\mathbf{1}$ is a vector with each element taking value one.
	
	\item The set of all unstable eigenvalues of $V$ is denoted as $\Lambda_{U}(V) \overset{def}{=} \{\lambda : |\lambda| \geq 1, det(V-\lambda I_n)=0 \}$.
	
	\item For an eigenvalue, $\lambda$, of a matrix $V$, the algebraic and geometric multiplicities are denoted by $a_{V}(\lambda)$ and $g_{V}(\lambda)$, respectively. For notational convenience, we sometimes denote an eigenvector, $v$, of $V$ corresponding to $\lambda \in sp(V)$ as $v \in Null(V - \lambda I)$.
\end{itemize}

The paper is organized as follows. In Section \ref{section_problem}, we describe the structure of an augmented observer and associated error dynamics. In Section \ref{section_main}, which presents the main result of our work, we focus on proving the existence of a network of observers, where the state estimate of each observer asymptotically converges to the state of the LTI system. Finally, discussions and future directions are presented.

\section{Problem Formulation} \label{section_problem}
In this section, we introduce the structure of our augmented observer used throughout this paper. First, we recall that given parameters to the design of a network of observers are $A$, $\{C_i\}_{i \in \{1, \cdots, m\}}$, and $\{\mathcal{N}_i\}_{i \in \{1, \cdots, m\}}$, where $\mathcal{N}_i$ is a set of neighbors of observer $i$ including observer $i$ itself.

This observer is called \textit{augmented} because the dimension of its state is larger than that of the LTI system. Thus the dimension of the augmented observer is no smaller than that of the corresponding Luenberger observer. As it is explained in the following section, with the augmented observers we are able to adopt the schemes used in the design of a dynamic compensator for LTI systems \cite{wang1973_tac, davison1990_tac}. Considering this point, we present the augmented observer as follows.
\begin{equation} \label{eq_augmented_estimator}
\begin{split}
	\hat{x}_i (k+1) &= \sum_{j \in \mathcal{N}_i} \left[ \mathbf{w}_{ij} A \underbrace{\hat{x}_j (k)}_\text{state estimate} + \mathbf{H}_{ij} \underbrace{ \left ( y_j (k) - C_{j} \hat{x}_j (k) \right )}_\text{measurement residual} \right. \\ & \quad \quad \left. + \mathbf{S}_{ij} \underbrace{z_j (k)}_\text{augmented state} \right] \\
	z_i(k+1) &= \sum_{j \in \mathcal{N}_i} \left[ \mathbf{R}_{ij} z_j (k) + \mathbf{Q}_{ij} \left ( y_j(k) - C_j \hat{x}_j (k) \right ) \right]
\end{split}
\end{equation}
where $\mathbf{H}_{ij} \in \mathbb{R}^{n \times r_j}$, $\mathbf{S}_{ij} \in \mathbb{R}^{n \times \mu_j}$, $\mathbf{Q}_{ij} \in \mathbb{R}^{\mu_i \times r_j}$, and $\mathbf{R}_{ij} \in \mathbb{R}^{\mu_i \times \mu_j}$, and $\mu_i$ is the dimension of the augmented state $z_i$. We refer $\mathbf{W}$ to a weight matrix, where $i,j$-th element of $\mathbf{W}$ is defined as $[\mathbf{W}]_{ij} = \mathbf{w}_{ij}$, and $\mathbf{H}_{ij}$, $\mathbf{S}_{ij}$, $\mathbf{Q}_{ij}$, and $\mathbf{R}_{ij}$ to gain parameters. In this work, we focus on finding the weight matrix and gain parameters such that $||\hat{x}_i(k) - x(k)|| \rightarrow 0$ as $k \rightarrow \infty$ for all $i \in \{1, \cdots, m\}$.

that achieve our objective while satisfying a given sparsity constraint.

\begin{figure} [t]
	\centering
	\includegraphics[width=0.5\textwidth]{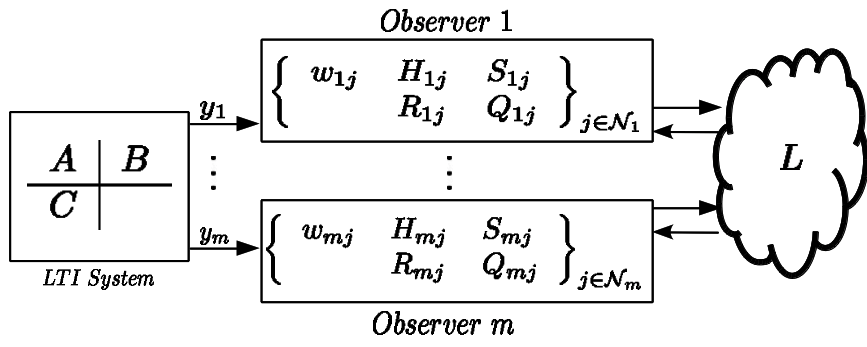}
	\caption {A framework for distributed estimation}
	\label{figure_framework}
\end{figure}

Suppose, for each $i \in \{1, \cdots, m\}$, $\sum_{j \in \mathcal{N}_i} \mathbf{w}_{ij} = 1$, then the error dynamics of \eqref{eq_augmented_estimator} can be written as
\begin{equation} \label{eq_error_dynamics}
\begin{split}
	\epsilon_i (k+1) &= \sum_{j \in \mathcal{N}_i} [ \left (\mathbf{w}_{ij} A - \mathbf{H}_{ij} C_{j} \right ) \epsilon_j(k) - \mathbf{S}_{ij} z_j (k) ] \\
	z_i(k+1) &= \sum_{j \in \mathcal{N}_i} \left [ \mathbf{R}_{ij} z_j (k) + \mathbf{Q}_{ij} C_j \epsilon_j (k) \right ]
\end{split}
\end{equation}
where $\epsilon_i(k) \overset{def}{=}  x(k) - \hat{x}_i(k)$. Throughout this paper, we assume that $\mathbf{w}_{ij} = \mathbf{w}_{ji}$ for all $i,j \in \{1, \cdots, m\}$.

\begin{remark}
	Notice that the error dynamics \eqref{eq_error_dynamics} is similar to the state space representation treated in the decentralized control problem \cite{wang1973_tac, davison1990_tac}. Hence, once the weight matrix is found and the existence condition for a network of observers is satisfied then the gain parameters can be computed by the result in \cite{wang1973_tac, davison1990_tac}.
\end{remark}

In this work, we are interested in the design of augmented observers under the choice of $\mathbf{W} = I_m - \alpha L$, where $\alpha > 0$ and $L$ is the Laplacian matrix of an underlying communication graph. It is beneficial to derive $\mathbf{W}$ from $L$ since $\mathbf{W}$ inherits the spectral property, i.e. eigenvalues and eigenvectors, of $L$ \footnote{One may have a better choice of $\mathbf{W}$ which, for instance, gives smaller eigenvalues by solving an optimization problem \cite{xiao2004_scl}.}.

Notice that we may collectively rewrite \eqref{eq_error_dynamics} as follows.
\begin{equation} \label{eq_error_dynamics02}
	\begin{split}
		\begin{pmatrix}
			\epsilon (k+1) \\
			z(k+1)
		\end{pmatrix} = 
		\begin{pmatrix}
			\mathbf{W} \otimes A - \mathbf{H} \bar{C} & -\mathbf{S} \\
			\mathbf{Q} \bar{C} & \mathbf{R}
		\end{pmatrix}
		\begin{pmatrix}
			\epsilon (k) \\
			z(k)
		\end{pmatrix}
	\end{split}
\end{equation}
where $\bar{C} \overset{def}{=} \left( e_1 \otimes C_1^T, \cdots, e_m \otimes C_m^T \right)^T$. Let $L$ be the Laplacian matrix, which describes the underlying communication graph, then \eqref{eq_error_dynamics02} satisfies a given sparsity constraint, i.e. $$\mathcal{P}(\mathbf{H}) \leq \mathcal{P}(L), \mathcal{P}(\mathbf{S}) \leq \mathcal{P}(L), \mathcal{P}(\mathbf{Q}) \leq \mathcal{P}(L), \mathcal{P}(\mathbf{R}) \leq \mathcal{P}(L)$$


\section{Main Result} \label{section_main}
The main theorem of this paper is presented in this section. First, we provide the main theorem without proof, then we state and explain supporting propositions and lemmas followed by the proof of the theorem. To state the theorem, we recall the following notation: let $L$ be the Laplacian matrix of a communication graph, and $\lambda_2$ and $\lambda_m$ be the second smallest and the largest eigenvalues of $L$, respectively.

\begin{thm} [Main Result] \label{thm_existence_gain}
	Suppose a detectable LTI system described by $(A, C)$ is given and it holds that
	$$\rho(A) < \frac{\lambda_m + \lambda_2}{\lambda_m - \lambda_2}$$
	where $\rho(A)$ is the spectral radius of $A$. For the weight matrix $\mathbf{W}$, choose  
	$$\mathbf{W} = I_m - \alpha L$$
	where $\alpha$ satisfies 
	$$\lambda_2^{-1} \left ( 1 - \rho^{-1}(A) \right ) < \alpha < \lambda_m^{-1} \left ( 1 + \rho^{-1}(A) \right )$$
	Then, for each $i,j \in \{1, \cdots, m\}$ and for some $\mu_j \in \mathbb{Z}_{+}$, there exist the gain parameters $\mathbf{H}_{ij} \in \mathbb{R}^{n \times r_j}$, $\mathbf{S}_{ij} \in \mathbb{R}^{n \times \mu_j}$, $\mathbf{Q}_{ij} \in \mathbb{R}^{\mu_i \times r_j}$, and $\mathbf{R}_{ij} \in \mathbb{R}^{\mu_i \times \mu_j}$ that satisfy a given sparsity constraint, i.e. 
	$$\mathcal{P}(\mathbf{H}) \leq \mathcal{P}(L), \mathcal{P}(\mathbf{S}) \leq \mathcal{P}(L), \mathcal{P}(\mathbf{Q}) \leq \mathcal{P}(L), \mathcal{P}(\mathbf{R}) \leq \mathcal{P}(L)$$
	and stabilize the error dynamics \eqref{eq_error_dynamics02}.
\end{thm}

\begin{remark}
	As is stated and proved later, the choice of $\alpha \in \left( \lambda_2^{-1} \left ( 1 - \rho^{-1}(A) \right ), \lambda_m^{-1} \left ( 1 + \rho^{-1}(A) \right ) \right)$ is allowed if and only if $\rho(A) < \frac{\lambda_m + \lambda_2}{\lambda_m - \lambda_2}$. This result only gives a sufficient condition; hence even if $\rho(A) < \frac{\lambda_m + \lambda_2}{\lambda_m - \lambda_2}$ does not hold, there may exists a network of observers.
\end{remark}

\subsection{Preliminary Results and Properties, subject to $\mathbf{W} = I_m - \alpha L$}
Our choice of $\mathbf{W} = I_m - \alpha L$ is adopted from the consensus literature (for instance, see \cite{olfati-saber2007_IEEEproceedings}). It is beneficial to use this relation since the existence condition can be represented by the Laplacian matrix of the underlying communication graph. The following proposition supports this argument.

\begin{prop} \label{prop_choosing_W}
	Suppose the Laplacian matrix $L$ is irreducible. Choose $\mathbf{W} = I_m - \alpha L$, where $\alpha > 0$. If the spectral radius of $A$, $\rho(A)$, satisfies
	$$\rho(A) < \frac{\lambda_m + \lambda_2} {\lambda_m - \lambda_2}$$
	then for
	$$\alpha \in \left ( \lambda_2^{-1} \left ( 1 - \rho^{-1}(A) \right ), \lambda_m^{-1} \left ( 1 + \rho^{-1}(A) \right ) \right )$$
	it holds that each eigenvalue of $\mathbf{W} \otimes A$ is either in $\Lambda_U(A)$ or inside the unit circle. Furthermore, the algebraic multiplicity of $\lambda \in \Lambda_U(\mathbf{W} \otimes A)$ is equal to $a_A (\lambda)$.
\end{prop}

\begin{proof}
	Since $L$ is irreducible,
	$$0 = \lambda_1 < \lambda_2 \leq \cdots \leq \lambda_m$$
	Set $\mathbf{W} = I_m - \alpha L$, then the eigenvalues of $\mathbf{W}$ satisfy 
	$$1 = 1- \alpha \lambda_1 > 1- \alpha \lambda_2 \geq \cdots \geq 1- \alpha \lambda_m$$

	To achieve our goal, we want to have
	$$1 >  \rho(A) (1- \alpha \lambda_2) \geq \cdots \geq \rho(A) (1- \alpha \lambda_m) > -1$$
	This gives us an inequality: 
	\begin{equation} \label{eq_inequality_alpha}
		\lambda_2^{-1} \left ( 1 - \rho^{-1}(A) \right ) < \alpha < \lambda_m^{-1} \left ( 1 + \rho^{-1}(A) \right )
	\end{equation}
	We can see that such $\alpha$ exists if $\rho(A) < \frac{\lambda_m + \lambda_2} {\lambda_m - \lambda_2}$.
	
	Notice that \cite{horn_topics_in_matrix_analysis}
	$$a_{\mathbf{W} \otimes A} (\lambda) = \sum_{(\lambda_\mathbf{W}, \lambda_A) \in \{(\lambda_\mathbf{W}, \lambda_A): \lambda = \lambda_\mathbf{W} \lambda_A\}} a_{\mathbf{W}}(\lambda_\mathbf{W}) a_{A}(\lambda_A)$$
	For $\lambda \in \Lambda_U(\mathbf{W} \otimes A)$, if \eqref{eq_inequality_alpha} holds then it is true that $\lambda = \lambda_\mathbf{W} \lambda_A$ with $\lambda_\mathbf{W} \in sp(\mathbf{W})$ and $\lambda_A \in sp(A)$ if and only if $\lambda_\mathbf{W} = 1$. Since $\mathbf{W}$ is irreducible and symmetric, it is true that $a_{\mathbf{W} \otimes A}(\lambda) = a_{A}(\lambda)$ for $\lambda \in \Lambda_U(\mathbf{W} \otimes A)$.
\end{proof}

\begin{remark}
	The following are two extreme cases of $\lambda_2 = \lambda_m$ and $\lambda_2 = 0$.
	\begin{enumerate}
		\renewcommand{\labelenumi}{\Roman{enumi}.}
		\item If the communication graph is complete, i.e. $\lambda_2 = \lambda_m$, then for any $\rho(A)$, we can select $\alpha$ such that $\mathbf{W} \otimes A$ has eigenvalue, $\lambda$, either in $\Lambda_U(A)$ with the algebraic multiplicity equal to $a_A(\lambda)$ or inside the unit circle.

		\item If the communication graph is not connected then $\lambda_2 = 0$ and the inequality condition stated in Proposition \ref{prop_choosing_W} is not valid unless $\rho(A) < 1$. Also by the property of the Kronecker product, the algebraic multiplicity of $\lambda \in \Lambda_U(W \otimes A) \left( \cap \Lambda_U(A) \right)$ is larger than $a_A(\lambda)$.
	\end{enumerate}
\end{remark}

\begin{remark} [Relation with a Consensus Problem]
	Consider the following consensus problem.
	\begin{equation}
		x(k+1) = \mathbf{W} x(k) = (I_m - \alpha L) x(k)
	\end{equation}
	 Notice that consensus is achieved if $\lambda_2 \neq 0$ \cite{xiao2004_scl}, i.e.
	$$\lim_{k \to \infty} x(k) = \mathbf{1} \cdot (1/m \cdot \mathbf{1}^T x(0))$$
	We can see that the condition that $\lambda_2 \neq 0$ is equivalent to $\rho(A) < \frac{\lambda_m + \lambda_2} {\lambda_m - \lambda_2}$ with $\rho(A) = 1$.
\end{remark}

\begin{remark}
	The choice of $\alpha$ in $\mathbf{W} = I_m - \alpha L$ affects the convergence rate of the error dynamics \eqref{eq_error_dynamics02}. Using the main theorem, under a proper choice of the gain parameters $\mathbf{H}$, $\mathbf{S}$, $\mathbf{Q}$, and $\mathbf{R}$, we may show that for every $\delta > 0$, there exist $c > 0$ and $\bar{\lambda} < r \leq \bar{\lambda} + \delta$ such that 
	$$\left\| \begin{pmatrix} \epsilon(k) \\ z(k) \end{pmatrix} \right\| \leq c \cdot r^k$$
	for all $k \in \mathbb{N}$, where $\bar{\lambda} = max_{\lambda_L \in sp(L) \setminus \{1\}} |(1 - \alpha \lambda_L) \cdot \rho(A)|$.
\end{remark}

Using Proposition \ref{prop_choosing_W}, we state the following proposition.
\begin{prop} \label{prop_detectability}
	If a pair $(A, C)$ is detectable and it holds that
	$$\rho(A) < \frac{\lambda_m + \lambda_2} {\lambda_m - \lambda_2}$$
	Then $(\mathbf{W} \otimes A, \bar{C})$ with $\mathbf{W} = I_m - \alpha L$ is detectable for 
	$$\alpha \in \left ( \lambda_2^{-1} \left ( 1 - \rho^{-1}(A) \right ), \lambda_m^{-1} \left ( 1 + \rho^{-1}(A) \right ) \right )$$
\end{prop}

	Before proving the above proposition, we give a lemma that describes the structure of eigenvectors of $\mathbf{W} \otimes A$ which correspond to eigenvalues in $\Lambda_U(\mathbf{W} \otimes A)$ in terms of eigenvectors of $\mathbf{W}$ and $A$. In general, an eigenvector of the Kronecker product of two matrices may not be the Kronecker product of two eigenvectors of the individual matrices. However, the following lemma shows that under our choice of $\alpha$ described in Proposition \ref{prop_detectability}, an eigenvector of $\mathbf{W} \otimes A $ corresponding to an eigenvalue in $\Lambda_U(\mathbf{W} \otimes A)$ can be written as the Kronecker product of eigenvectors of $\mathbf{W}$ and $A$.
	
\begin{lemma} \label{lemma_generalized_eigenvector}
	Suppose the following inequality holds:
	$$\rho(A) < \frac{\lambda_m + \lambda_2} {\lambda_m - \lambda_2}$$
	Then, under the choice of 
	$$\alpha \in \left ( \lambda_2^{-1} \left ( 1 - \rho^{-1}(A) \right ), \lambda_m^{-1} \left ( 1 + \rho^{-1}(A) \right ) \right )$$
	for  $\lambda \in \Lambda_U(\mathbf{W} \otimes A)$, a corresponding eigenvector, $v$, can be written as
	$$v = v_\mathbf{W} \otimes v_A$$
	where $v_\mathbf{W} \in Null(\mathbf{W}-I)$ and $v_A \in Null(A-\lambda I)$.	
\end{lemma}

\begin{proof}
	Let $\lambda_\mathbf{W}$ and $\lambda_A$ be eigenvalues of $\mathbf{W}$ and $A$, respectively. Then by our choice of $\mathbf{W} = I_m - \alpha L$ with $\alpha \in \left ( \lambda_2^{-1} \left ( 1 - \rho^{-1}(A) \right ), \lambda_m^{-1} \left ( 1 + \rho^{-1}(A) \right ) \right )$ and Proposition \ref{prop_choosing_W}, $\lambda \in \Lambda_U(\mathbf{W} \otimes A)$ can be written as $\lambda = \lambda_\mathbf{W} \lambda_A$ only if $\lambda_\mathbf{W} = 1$ and $\lambda_A = \lambda$. Since the algebraic multiplicity of $\lambda_\mathbf{W} = 1$ is one, it holds that $a_{\mathbf{W} \otimes A}(\lambda) = a_A(\lambda)$.

	We can write the set of all generalized eigenvectors of $A$ corresponding to $\lambda_A = \lambda$ as
	\begin{equation}
		\begin{split}
			\left \{ \xi_1, (A - \lambda_{A} I_n)\xi_1, \cdots, (A - \lambda_{A} I_n)^{m(\xi_1)}\xi_1, \cdots, \xi_{g_{A}(\lambda_{A})}, \right . \\ \left . (A - \lambda_{A} I_n)\xi_{g_{A}(\lambda_{A})}, \cdots, (A - \lambda_{A} I_n)^{m(\xi_{g_{A}(\lambda_{A})})}\xi_{g_{A}(\lambda_{A})} \right \}
		\end{split}
	\end{equation}
	where $m(\xi_i)$ is the largest integer such that $(A - \lambda_{A} I_n)^{m(\xi_i)}\xi_i$ is nonzero (see \cite{axler2004_linear_algebra} for details). We can see that the cardinality of the above set is equal to $a_A(\lambda_{A})$, and 
	$$\left \{ (A - \lambda_{A} I_n)^{m(\xi_1)}\xi_1, \cdots, (A - \lambda_{A} I_n)^{m(\xi_{g_{A}(\lambda_{A})})}\xi_{g_{A}(\lambda_{A})}\right \}$$ 
	becomes the set of all eigenvectors of $A$ corresponding to $\lambda_{A}$.

	We claim that for $\lambda = \lambda_A$, 
	\begin{equation} \label{eq_gen_eig_01}
		v_\mathbf{W} \otimes (A - \lambda_{A} I_n)^{l} \xi_i
	\end{equation}
	is a generalized eigenvector of $\mathbf{W} \otimes A$ corresponding to $\lambda$ for $l \in \{0, \cdots, m(\xi_i)\}$ and $i \in \{1, \cdots, g_{A}(\lambda_{A})\}$, where $v_\mathbf{W} \in Null (\mathbf{W}-I)$. To prove this claim, consider the following.
	
	For $0 \leq l < m(\xi_i)$,
	\begin{equation}
		\begin{split}
			&\left ( \mathbf{W} \otimes A - \lambda_A I_{m \cdot n} \right ) \left (v_\mathbf{W} \otimes (A - \lambda_{A} I_n)^{l} \xi_i \right ) \\
			&= v_\mathbf{W} \otimes A(A - \lambda_{A} I_n)^{l} \xi_i - v_\mathbf{W} \otimes \lambda_{A}(A - \lambda_{A} I_n)^{l} \xi_i \\
			&= v_\mathbf{W} \otimes (A - \lambda_{A} I_n)^{l+1} \xi_i \neq 0
		\end{split}
	\end{equation}
	The first equality comes from the fact that $\mathbf{W} \cdot v_\mathbf{W} = v_\mathbf{W}$. For $l = m(\xi_i)$, \eqref{eq_gen_eig_01} becomes an eigenvector of $\mathbf{W} \otimes A$ since $(A - \lambda_{A} I_n)^{m(\xi_i)+1}\xi_i = 0$.

	Note that let $v_1, v_2 \in \mathbb{R}^m$ and $w_1, w_2 \in \mathbb{R}^n$ be nonzero vectors, then $v_1 \otimes w_1$ and $v_2 \otimes w_2$ are linearly independent if and only if either $v_1$ and $v_2$ or $w_1$ and $w_2$ are linearly independent. Hence 
	$$\{v_\mathbf{W} \otimes (A - \lambda_{A} I_n)^{l} \xi_i\}_{l \in \{0, \cdots, m(\xi_i)\}, ~ i \in \{1, \cdots, g_A(\lambda_{A})\}}$$ 
	is the set of all generalized eigenvectors of $\mathbf{W} \otimes A$ corresponding to $\lambda = \lambda_{A}$. In particular, 
	$$\{v_\mathbf{W} \otimes (A - \lambda_{A} I_n)^{m(\xi_i)} \xi_i\}_{i \in \{1, \cdots, g_A(\lambda_{A})\}}$$
	becomes the set of all eigenvectors of $\mathbf{W} \otimes A$ corresponding to $\lambda = \lambda_A$. This proves our claim.
\end{proof}

\begin{proof} [The proof of Proposition \ref{prop_detectability}]
	To show that $(\mathbf{W} \otimes A, \bar{C})$ is detectable, we only need to show that for an unstable eigenvalue of $\mathbf{W} \otimes A$, its corresponding eigenvector, $v$, satisfies $\bar{C} v \neq 0$. By Lemma \ref{lemma_generalized_eigenvector}, for $\lambda \in \Lambda_U(\mathbf{W} \otimes A)$, its corresponding eigenvector, $v$, is of the form $v = v_\mathbf{W} \otimes v_{A}$, where $v_\mathbf{W} \in Null(\mathbf{W} - I)$ and $v_{A} \in Null(A - \lambda I)$. Since $\mathbf{W}$ is a stochastic matrix, the eigenvector of $\mathbf{W}$ corresponding to $\lambda_\mathbf{W} = 1$ is $1/\sqrt{m} \cdot \mathbf{1}$. Thus, we obtain
	\begin{equation}
		\bar{C} v = 1/\sqrt{m} \cdot \bar{C} (\mathbf{1}  \otimes v_{A}) = 1/\sqrt{m} \cdot \begin{pmatrix}
															C_1 v_{A} \\
															C_2 v_{A} \\
															\vdots \\
															C_m v_{A}
														\end{pmatrix}
	\end{equation}
	Since $(A, C)$ is detectable, $\bar{C} v$ is nonzero. This proves the statement that the pair $(\mathbf{W} \otimes A, \bar{C})$ is detectable.
\end{proof}

In a view of output feedback, if $(\mathbf{W} \otimes A, \bar{C})$ is not detectable, it is not possible to find the gain parameters $\mathbf{H}$, $\mathbf{S}$, $\mathbf{Q}$, and $\mathbf{R}$ which stabilize \eqref{eq_error_dynamics02}. In this context, if $(A, C)$ is not detectable then $(\mathbf{W} \otimes A, \bar{C})$ is also not detectable; hence there exist no gain parameters that stabilize \eqref{eq_error_dynamics02}.

\begin{ex} \label{ex_connectedness}
	This example shows the reason why we need a connected communication graph, even if $(A, C)$ is detectable. Suppose the system matrices $A, C$ and the weight matrix $\mathbf{W}$ are given as follows.
	\begin{equation}
		\begin{split}
			A = I_3, C_1 = [1~0~0],  C_2 = [0~1~0], C_3 = [0~0~1], \\
			\mathbf{W} = I_3- \alpha L = \begin{pmatrix}
							1-\alpha & \alpha & 0 \\
							\alpha & 1-\alpha & 0 \\
							0 & 0 & 1
						\end{pmatrix}
		\end{split}			
	\end{equation}

	We can see that $(A, C)$ is observable (thus detectable). However, for the eigenvector $v = (0~0~1)^T \otimes (1~1~0)^T$ of $\mathbf{W} \otimes A$ corresponding to an eigenvalue at $1$, it holds that $\bar{C} v = 0$. Hence the system $(\mathbf{W} \otimes A, \bar{C})$ is not detectable.
\end{ex}

\subsection{Gain Parameters: $\mathbf{H}$, $\mathbf{Q}$, $\mathbf{R}$, and $\mathbf{S}$}
Here we study the choice of the gain parameters of the augmented observer \eqref{eq_augmented_estimator} subject to $\mathbf{W} = I_m - \alpha L$. This is done by writing our formulation into a form that allows us to apply results from decentralized control literature \cite{wang1973_tac, davison1990_tac}. To apply the results from decentralized control literature, we need to verify, in a view of output feedback, whether unstable modes (eigenvalues) in our formulation can be stabilizable \footnote{We want to make a note that the notion of these unstable modes of our network of observers, which is subject to a sparsity constraint, is slightly different from that of the (standard) Luenberger observer due to the sparsity constraint. For more detail, we refer readers to \cite{gong1992_ieee_tac, anderson1981_automatica, davison1983_automatica}}. An algebraic way to check this condition is presented in Theorem 1 of \cite{davison1990_tac}. Using this condition, we explicitly state and prove a sufficient condition under which each unstable mode of the error dynamics is stabilizable. This draws a direct relation between the detectability of $(A, C)$ and the existence of the gain parameters for the observers under a sparsity constraint.

The following notation is additionally defined for convenience.
\begin{itemize}
	\item Given $\{V_i\}_{i \in \{1, \cdots, m\}}$, $diag \left( \{V_i\}_{i \in \{1, \cdots, m\}} \right)$ is a block diagonal matrix.
	
	\item Given $\{V_{i,j}\}_{j \in \{1, \cdots, m\}}$ and $\mathcal{N} = \{j_1, \cdots, j_s\} \subseteq \{1, \cdots, m\}$, we define $V_{i,\mathcal{N} \downarrow} \overset{def}{=} \begin{pmatrix} V_{i,j_1} \\ \vdots \\ V_{i,j_s} \end{pmatrix}$ and $E_{\mathcal{N}} = \left( e_{j_1}, \cdots, e_{j_s}\right)$.
\end{itemize}

First, observe that $\mathbf{W} \otimes A - \mathbf{H} \bar{C}$ can be written as
\begin{equation}
	\mathbf{W} \otimes A - \mathbf{H} \bar{C} = \mathbf{W} \otimes A - \sum_{i=1}^{m} \bar{B}_i \bar{\mathbf{H}}_i \bar{C}_i
\end{equation}
where $\bar{B}_i \overset{def}{=} E_{\mathcal{N}_i} \otimes I_n$, $\bar{\mathbf{H}}_i \overset{def}{=} \mathbf{H}_{i,\mathcal{N}_i \downarrow}$, and $\bar{C}_i \overset{def}{=} e_i^T \otimes C_i$. Notice that unlike $\mathbf{H}$, there is no sparsity constraint imposed on $\bar{\mathbf{H}}_i$ for each $i$.

Letting 
	$$\bar{B} = \left( \bar{B}_1, \cdots, \bar{B}_m \right), \quad \bar{C} = \left( \bar{C}_1^T, \cdots, \bar{C}_m^T \right)^T$$
	$$\bar{\mathbf{H}} = diag \left ( \{ \bar{\mathbf{H}}_i \}_{i \in \{1, \cdots, m\}} \right ), \bar{\mathbf{S}} = diag \left ( \{ \bar{\mathbf{S}}_i \}_{i \in \{1, \cdots, m\}} \right )$$
with $\bar{\mathbf{H}}_i \overset{def}{=} \mathbf{H}_{i,\mathcal{N}_i \downarrow}$ and $\bar{\mathbf{S}}_i \overset{def}{=} \mathbf{S}_{i,\mathcal{N}_i \downarrow}$, we can rewrite \eqref{eq_error_dynamics02} as follows.
\begin{equation} \label{eq_error_dynamics03}
	\begin{split}
		\begin{pmatrix}
			\epsilon (k+1) \\
			z(k+1)
		\end{pmatrix} = 
		\begin{pmatrix}
			\mathbf{W} \otimes A - \bar{B} \bar{\mathbf{H}} \bar{C} & -\bar{B} \bar{\mathbf{S}} \\
			\mathbf{Q} \bar{C} & \mathbf{R}
		\end{pmatrix}
		\begin{pmatrix}
			\epsilon (k) \\
			z(k)
		\end{pmatrix}
	\end{split}
\end{equation}

To prove the main theorem, we first find the Kalman decomposition for both $(A, C_i)$ and $(\mathbf{W} \otimes A, \bar{B}_i, \bar{C}_i)$, and explicitly show that the observable part of $(A, C_i)$ is actually contained in the controllable and observable part of $(\mathbf{W} \otimes A, \bar{B}_i, \bar{C}_i)$. In the preceding propositions, we have seen that $\Lambda_U(A) = \Lambda_U(\mathbf{W} \otimes A)$ with the same algebraic multiplicity, i.e. $a_A(\lambda) = a_{\mathbf{W} \otimes A}(\lambda) \text{ for } \lambda \in \Lambda_U(A)$. Hence if $(A, C)$ is detectable then every unstable mode of $\mathbf{W} \otimes A$ is contained in the controllable and observable part of  $(\mathbf{W} \otimes A, \bar{B}_i, \bar{C}_i)$ for some $i \in \{1, \cdots, m\}$. Applying the result in \cite{wang1973_tac, brasch1970_tac}, we verify that every unstable mode of $\mathbf{W} \otimes A$ can be placed inside the unit circle by a proper choice of the gain parameters $\mathbf{H}$, $\mathbf{Q}$, $\mathbf{R}$, and $\mathbf{S}$. This proves the existence of the observers.

The following lemmas provide a basis for proving our main theorem.

\begin{lemma} \label{lemma_mode_stabilizability}
	Suppose the pair $(A, C)$ is detectable and 
	$$\rho(A) < \frac{\lambda_m + \lambda_2} {\lambda_m - \lambda_2}$$
	holds. Then, under the choice of 
	$$\alpha \in \left ( \lambda_2^{-1} \left ( 1 - \rho^{-1}(A) \right ), \lambda_m^{-1} \left ( 1 + \rho^{-1}(A) \right ) \right ),$$
	taking the multiplicity into account, it is true that every unstable mode of $\mathbf{W} \otimes A$ is in the controllable and observable part of $(\mathbf{W} \otimes A, \bar{B}_i, \bar{C}_i)$ for some $i \in \{1, \cdots, m\}$.
\end{lemma}

\begin{proof}
	The proof is given in Appendix \ref{subsection_lemma_mode_stabilizability}.
\end{proof}

\begin{lemma} \label{lemma_fixed_mode}
	Given $\mathcal{A} \in \mathbb{R}^{n \times n}$, $\left\{ \mathcal{B}_i \right\}_{i \in \{1, \cdots, m\}}$, and $\left\{ \mathcal{C}_i \right\}_{i \in \{1, \cdots, m\}}$, consider a partition, $\{i_1, \cdots, i_l\}$ and $\{i_{l+1}, \cdots, i_m\}$, of the set $\{1, \cdots, m\}$. For $\lambda \in sp(\mathcal{A})$, let $k_\mathcal{B} = dim \left( Null \left( \begin{pmatrix} \mathcal{A} - \lambda I & \mathcal{B}_{i_1} & \cdots & \mathcal{B}_{i_l}  \end{pmatrix}^T \right) \right)$. Suppose there exist $k_\mathcal{B}$ linearly independent eigenvectors, $\{\eta_j\}_{j \in \{1, \cdots, k_\mathcal{B} \}}$, of $\mathcal{A}$ corresponding to $\lambda$ such that 
	$$\eta_j \notin Null \begin{pmatrix} \mathcal{C}_{i_{l+1}} \\ \vdots \\ \mathcal{C}_{i_{m}} \end{pmatrix}, ~ \forall j \in \{1, \cdots, k_\mathcal{B}\}$$
	Then it holds that 
	$$rank \begin{pmatrix} \mathcal{A} - \lambda I & \mathcal{B}_{i_1} & \cdots & \mathcal{B}_{i_l} \\ \mathcal{C}_{i_{l+1}} & 0 & \cdots & 0 \\ \vdots & \vdots & \ddots & \vdots \\ \mathcal{C}_{i_{m}} & 0 & \cdots & 0 \end{pmatrix} \geq n$$
	where $n$ is the dimension of the matrix $\mathcal{A}$.
\end{lemma}
\begin{proof}
	The proof is given in Appendix \ref{subsection_lemma_fixed_mode}.
\end{proof}

\begin{proof} [The proof of Theorem \ref{thm_existence_gain}]
	In this proof, using Lemma \ref{lemma_mode_stabilizability} and \ref{lemma_fixed_mode}, we show the existence of the gain parameters which stabilize the error dynamics \eqref{eq_error_dynamics02}.
	
	Let $\{i_1, \cdots, i_l\}$ and $\{i_{l+1}, \cdots, i_m\}$ be a partition of the set $\{1, \cdots, m\}$. Then by Lemma \ref{lemma_mode_stabilizability}, for the (left) eigenvector $q$ of $\mathbf{W} \otimes A$ corresponding to $\lambda \in \Lambda_U(\mathbf{W} \otimes A)$, if $q^T \begin{pmatrix} \bar{B}_{i_1} & \cdots & \bar{B}_{i_l} \end{pmatrix} = 0$ then there exists an associated (right) eigenvector, $p$, such that $p \notin Null \begin{pmatrix} \bar{C}_{i_{l+1}} \\ \vdots \\ \bar{C}_{i_{m}} \end{pmatrix}$; otherwise there exists $\lambda \in \Lambda_U(\mathbf{W} \otimes A)$ that is not contained in the controllable and observable part of $(\mathbf{W} \otimes A, \bar{B}_i, \bar{C}_i)$ for any $i \in \{1, \cdots, m\}$.
	
	Notice that this holds for any partition $\{i_1, \cdots, i_l\}$ and $\{i_{l+1}, \cdots, i_m\}$, and for any such partition, by Lemma \ref{lemma_fixed_mode}, it is true that 
	$$rank \begin{pmatrix} \mathbf{W} \otimes A - \lambda I & \bar{B}_{i_1} & \cdots & \bar{B}_{i_l} \\ \bar{C}_{i_{l+1}} & 0 & \cdots & 0 \\ \vdots & \vdots & \ddots & \vdots \\ \bar{C}_{i_{m}} & 0 & \cdots & 0 \end{pmatrix} \geq n \cdot m$$
	By the result in \cite{wang1973_tac, brasch1970_tac}, there exist $\bar{\mathbf{H}}$, $\bar{\mathbf{S}}$, $\mathbf{Q}$, and $\mathbf{R}$ that stabilize the error dynamics \eqref{eq_error_dynamics03} and satisfies
	$$\mathcal{P}(\bar{\mathbf{H}}) \leq \mathcal{P}(L), \mathcal{P}(\bar{\mathbf{S}}) \leq \mathcal{P}(L), \mathcal{P}(\mathbf{Q}) \leq \mathcal{P}(L), \mathcal{P}(\mathbf{R}) \leq \mathcal{P}(L)$$
	This proves the existence of the gain parameters for the augmented observers.
\end{proof}

\begin{remark} \label{remark_thm_proof}
	In Theorem \ref{thm_existence_gain}, we have seen that if the inequality condition $\rho(A) < \frac{\lambda_m + \lambda_2}{\lambda_m - \lambda_2}$ holds then the existence of the gain parameters follows. However as the proof of our main theorem suggests it is possible to find more general condition for the existence of a network of observers for this distributed estimation problem.
\end{remark}

\begin{remark} \label{remark_sparsity_constraint}
	We want to mention that even if the result in \cite{wang1973_tac, brasch1970_tac} gives $\bar{\mathbf{H}}, \bar{\mathbf{S}}, \mathbf{Q}, \mathbf{R}$ which satisfy the sparsity constraint, for $\mathbf{Q}$ and $\mathbf{R}$, it only holds that $\mathcal{P}(\mathbf{Q}) \leq \mathcal{P}(I_m) < \mathcal{P}(L)$ and $\mathcal{P}(\mathbf{R}) \leq \mathcal{P}(I_m) < \mathcal{P}(L)$, i.e. $\mathbf{Q}$ and $\mathbf{R}$ are block-diagonal.
\end{remark}

\section{Discussions and Future Work} \label{section_discussions}
In this paper, we study the structure of a network of observers of the form \eqref{eq_augmented_estimator} for estimating the state of a LTI system described in \eqref{state_equation}. The existence condition for such observers is characterized by the spectral radius of $A$ and the eigenvalues of the Laplacian matrix of the underlying communication graph. In particular, we show that if it holds that $\rho(A) < \frac{\lambda_m + \lambda_2} {\lambda_m - \lambda_2}$ then the state estimate of each observer asymptotically converges to the state of the LTI system. In other words, let $\hat{x}_i(k)$ and $x(k)$ be the state estimate by observer $i$ and the state of the LTI system, respectively. If the existence condition is satisfied then $||\hat{x}_i(k) - x(k)|| \rightarrow 0$ as $k \rightarrow \infty$ for all $i \in \{1, \cdots, m\}$.

However, some parts of the conditions, given and proven in this paper, may not be strict and the choice of parameters is not optimal. For example, when choosing the weight matrix $\mathbf{W}$, instead of finding $\mathbf{W}$ that minimizes its spectral radius while satisfying a sparsity constraint, we adopt $\mathbf{W} = I - \alpha L$. This choice of $\mathbf{W}$ can be improved by solving an optimization problem \cite{xiao2004_scl}. Also as mentioned in Remark \ref{remark_sparsity_constraint}, when applying the result of \cite{wang1973_tac, davison1990_tac}, the possible choice of the gain parameters is not fully utilized, especially $\mathbf{Q}$ and $\mathbf{R}$.

As future works, it may be interesting to find the parameters, which minimizes the relative effect of the noise on the estimation error. This problem is nontrivial because the sparsity constraint makes the minimization problem non-convex.

In addition, as our approach relies on a centralized computation of the parameters, designing an algorithm for distributed computation of the parameters must be an intriguing problem. This may allow individual observers to reconfigure their parameters under a switching communication graph.


\section{Appendix} \label{appendix}
\subsection{Proof of Lemma \ref{lemma_mode_stabilizability}} \label{subsection_lemma_mode_stabilizability} \label{subsection_lemma_fixed_mode}

\begin{lemma} \label{lemma_null_space_W}
	Suppose the (symmetric) weight matrix $W$ has eigenvalues inside the unit circle except at $1$ with the algebraic multiplicity one. Let $\bar{\mathbf{W}}_i \overset{def}{=}  \left ( e_i, ~  \mathbf{W} e_i, ~  \cdots, ~  \mathbf{W}^{m -1} e_i \right )^T$. Then $Null (\bar{\mathbf{W}}_i) \subseteq Null(\mathbf{1}^T)$.
\end{lemma}

\begin{proof}
	By the definition of the null space and the Cayley-Hamilton theorem, $Null(\bar{\mathbf{W}}_i) = \cap_{j=1}^{m} Null \left ( e_i^T \mathbf{W}^{j-1}\right ) = \cap_{j=1}^{\infty} Null \left ( e_i^T \mathbf{W}^{j-1}\right )$.

	Note that $\mathbf{W}$ is symmetric and $1/\sqrt{m} \cdot \mathbf{1}$ is a unique (unit) eigenvector, corresponding to $\lambda_\mathbf{W} = 1$. Also $\lambda_\mathbf{W} = 1$ is the only unstable mode of $\mathbf{W}$. Thus $\mathbf{W}^{j}$ converges to $1/m \cdot \mathbf{1} \cdot \mathbf{1}^T$ as $j \to \infty$. From this fact, we can see that $\left [ \cap_{j=1}^{\infty} Null \left ( e_i^T \mathbf{W}^{j-1}\right ) \right ]^{\perp} \supseteq \cup_{j=1}^{\infty} Range (\mathbf{W}^{j-1} e_i) \supseteq Range (\mathbf{1})$.
	
	Since $\mathbb{R}^m$ is a finite dimensional vector space, it is true that $Null(\bar{\mathbf{W}}_i) = \cap_{j=1}^{\infty} Null \left ( e_i^T \mathbf{W}^{j-1}\right ) \subseteq \left [ \cup_{j=1}^{\infty} Range \left ( \mathbf{W}^{j-1} e_i \right ) \right ]^{\perp} \subseteq \left [ Range (\mathbf{1})\right ]^{\perp} = Null \left ( \mathbf{1}^T \right )$. This proves our claim.
\end{proof}

\begin{proof} [Proof of Lemma \ref{lemma_mode_stabilizability}]
	Under the choice of $\mathbf{W} = I_m - \alpha L$ with $\alpha \in \left ( \lambda_2^{-1} \left ( 1 - \rho^{-1}(A) \right ), \lambda_m^{-1} \left ( 1 + \rho^{-1}(A) \right ) \right )$, it is true that only unstable modes of $W \otimes A$ are that of $A$, where the relation between their corresponding eigenvectors is explicitly shown in Lemma \ref{lemma_generalized_eigenvector}. By the property of the Kronecker product and Lemma \ref{lemma_generalized_eigenvector}, it is straightforward to see that the number of unstable modes of $\mathbf{W} \otimes A$ coincides with that of $A$, i.e. $g_{\mathbf{W} \otimes A}(\lambda) = g_{A} (\lambda)$ for $\lambda \in \Lambda_U(\mathbf{W} \otimes A)$.

	In this proof, we show that, by applying the Kalman decomposition method, the observable part of $(A, C_i)$ is contained in the controllable and observable part of $(\mathbf{W} \otimes A, \bar{B}_i, \bar{C}_i)$. Then, since $(A, C)$ is detectable, this proves our claim.

	For notational convenience, we define the following notation.
	\begin{itemize}
		\item Given $\{a_i\}_{i \in \{1, \cdots, n\}}$ and $\{j_1, \cdots, j_l\} \subseteq \{1, \cdots, n\}$, we define $a_{\{j_1, \cdots, j_l\}} \overset{def}{=} \left\{ a_{j_1}, \cdots, a_{j_l} \right\}$.
	\end{itemize}
	
	Let $k_i = dim \left ( Null(\mathcal{O}_i) \right )$, where $\mathcal{O}_i$ is the observability matrix of $(A, C_i)$. Then we can construct a $n \times n$ nonsingular matrix $P$ as follows.
	\[
		P = \left ( p_{\{1, \cdots, k_i\}}, p_{\{k_i+1, \cdots, n\}} \right )
	\]
	where $p_{\{1, \cdots, k_i\}} \subseteq Null(\mathcal{O}_i)$ and $p_{\{k_i+1, \cdots, n\}} \subseteq Range(\mathcal{O}^T_i)$. Note that we may choose column vectors in $P$ to be orthogonal to each other.

	Then by the Kalman decomposition, we obtain
	\begin{equation} \label{eq_kalman_decomp}
	\begin{split}
		\hat{A}^{(i)} &= P^{-1} A P = \begin{pmatrix}
								A_{NO}^{(i)} & A^{(i)}_{2} \\
								0 & A_{O}^{(i)}
							\end{pmatrix} \\
		\hat{C}^{(i)} &= C_iP = \left ( 0 ~ C_{O}^{(i)}\right )
	\end{split}
	\end{equation}
where $\left( A_O^{(i)}, C_O^{(i)}\right)$ is an observable part.

	Similarly, let $\bar{\mathcal{O}}_i$ be the observability matrix of $(\mathbf{W} \otimes A, \bar{B}_i, \bar{C}_i)$ and $l_i = dim \left ( Null \left ( \left ( e_i, \mathbf{W} e_i, \cdots, \mathbf{W}^{m -1} e_i \right ) ^T \right ) \right )$. Choose $\left ( \beta_{\{1, \cdots, l_i\}}, \beta_{\{l_i+1, \cdots, m\}} \right )$, where $\beta_{\{1, \cdots, l_i\}} \subseteq Null \left ( \left ( e_i, \mathbf{W} e_i, \cdots, \mathbf{W}^{m -1} e_i \right ) ^T \right )$, $\beta_{\{l_i+1, \cdots, m-1\}} \subseteq Null(\mathbf{1}^T) \cap Range \big ( e_i, \mathbf{W} e_i, \cdots, \mathbf{W}^{m -1} e_i \big )$, and $\beta_m = 1/\sqrt{m} \cdot \mathbf{1}$ \footnote{From Lemma \ref{lemma_null_space_W}, we can infer that $\beta_m = 1/\sqrt{m} \cdot \mathbf{1} \in Range \left( e_i, \mathbf{W}e_i, \cdots, \mathbf{W}^{m-1}e_i \right)$}. By Lemma \ref{lemma_null_space_W} and the structure of $\left ( e_i, \mathbf{W} e_i, \cdots, \mathbf{W}^{m -1} e_i \right ) ^T$, we may choose $\left ( \beta_{\{1, \cdots, l_i\}}, \beta_{\{l_i+1, \cdots, m\}} \right )$ to be nonsingular, and $\beta_m$ to be orthogonal to $\beta_j$ for $j \in \{1, \cdots, m-1\}$. Then we can construct a matrix $P_\mathbf{W}$ as
	\begin{equation*}
	\begin{split}
		P_\mathbf{W} = \big \{ \beta_{\{1, \cdots, l_i\}} \otimes P, \beta_{\{l_i+1, \cdots, m-1\}} \otimes P_1, \\ \eta_{\{1, \cdots, t\}}, \beta_m \otimes P_1, \xi_{\{1, \cdots, s\}}, \beta_m \otimes P_2 \big \}
	\end{split}
	\end{equation*}
	where $P_1 = p_{\{1, \cdots, k_i\}}$ and $P_2 = p_{\{k_i+1, \cdots, n\}}$.
	Here $\eta_{\{1, \cdots, t\}}$ and $\xi_{\{1, \cdots, s\}}$ are chosen properly so that $span \big \{ \beta_{\{1, \cdots, l_i\}} \otimes P, \beta_{\{l_i+1, \cdots, m-1\}} \otimes P_1, \eta_{\{1, \cdots, t\}}, \beta_m \otimes P_1 \big \} = Null (\bar{\mathcal{O}_i})$ and $P_\mathbf{W}$ is nonsingular. In particular, by our choice of $\beta_{\{1, \cdots, m\}}$ and $P$, each $\eta_j$ can be written as a linear combination of column vectors in $\beta_{\{l_i+1, \cdots, m\}} \otimes P_2$ for $j \in \{1, \cdots, t\}$. Similarly each $\xi_j$ can be written as a linear combination of column vectors in $\beta_{\{l_i+1, \cdots, m-1\}}\otimes P_2$ for $j \in \{1, \cdots, s\}$ \footnote{Since we already have $\beta_m \otimes P_2$ in $P_W$ and $span \left( \beta_m \otimes P_2 \right) \cap Null \left( \bar{\mathcal{O}}_i\right) = \emptyset$, $\xi_j$ is represented by columns in $\beta_{\{l_i+1, \cdots, m-1\}} \otimes P_2$.}.
	
	Using $P_\mathbf{W}$, we can describe the observable part of $(\mathbf{W} \otimes A, \bar{B_i}, \bar{C_i})$ as follows.
	\begin{equation} \label{eq_kalman_decomp_W}
		\begin{split}
			\hat{\bar{A}}^{(i)} = P_\mathbf{W}^{-1} (\mathbf{W} \otimes A) P_\mathbf{W} = \begin{pmatrix} \bar{A}_{NO}^{(i)} & \bar{A}_{2}^{(i)} \\ 0 & \bar{A}_O^{(i)} \end{pmatrix} \\
			\hat{\bar{C}}^{(i)} = (e_i^T \otimes C_i) P_\mathbf{W} = \left ( 0 ~ \bar{C}_{O}^{(i)} \right )
		\end{split}
	\end{equation}
	where $\left ( \bar{A}_O^{(i)}, \bar{C}_O^{(i)} \right )$ is an observable pair.
	
	Let $\bar{A}_{2}^{(i)} = \begin{pmatrix} \bar{A}_{2,1}^{(i)} & \bar{A}_{2,2}^{(i)} \\ \bar{A}_{2,3}^{(i)} & \bar{A}_{2,4}^{(i)} \end{pmatrix}$ and $\bar{A}_{O}^{(i)} = \begin{pmatrix} \bar{A}_{O,1}^{(i)} & \bar{A}_{O,2}^{(i)} \\ \bar{A}_{O,3}^{(i)} & \bar{A}_{O,4}^{(i)} \end{pmatrix}$. Then we obtain from \eqref{eq_kalman_decomp_W},
	\begin{equation*}
		\begin{split}
			&(\mathbf{W} \otimes A) \left ( \beta_m \otimes P_2 \right ) = \beta_m \otimes A P_2\\
			&= P_\mathbf{W} \begin{pmatrix} \bar{A}_{2,2}^{(i)} \\ \bar{A}_{2,4}^{(i)} \\ \bar{A}_{O,2}^{(i)} \\ \bar{A}_{O,4}^{(i)} \end{pmatrix} \\
			&= \left ( \beta_{\{1, \cdots, l_i\}} \otimes P, \beta_{\{l_i+1, \cdots, m-1\}} \otimes P_1, \eta_{\{1, \cdots, t\}} \right ) \bar{A}_{2,2}^{(i)} \\ & \qquad + (\beta_m \otimes P_1) \bar{A}_{2,4}^{(i)} + \xi_{\{1, \cdots, s\}} \bar{A}_{O,2}^{(i)} + ( \beta_m \otimes P_2 ) \bar{A}_{O,4}^{(i)} \\
		\end{split}
	\end{equation*}

	The first equality comes from the fact that $\mathbf{W} \beta_m = \beta_m$. Since $\hat{\bar{A}}^{(i)}$ is uniquely determined by $\mathbf{W} \otimes A$ and $P_\mathbf{W}$, we obtain $\bar{A}_{2,2}^{(i)} = \bar{A}_{O,2}^{(i)} = 0$ and $\bar{A}_{2,4}^{(i)} = A_{2}^{(i)}$, $\bar{A}_{O,4}^{(i)} = A_{O}^{(i)}$, where $A_{2}^{(i)}$ and $A_O^{(i)}$ are defined in \eqref{eq_kalman_decomp}. In addition we see that $\bar{C}_O^{(i)} = \left ( (e_i^T \otimes C_i) \xi_{\{1, \cdots, s\}},  (1/\sqrt{m})C_O^{(i)} \right )$. Hence the observable part of $(A, C)$ is contained in the observable part of $(\mathbf{W} \otimes A, \bar{B_i}, \bar{C_i})$.
	
	To show that modes of $A_O^{(i)}$ in $\hat{\bar{A}}^{(i)}$ are controllable in $(\mathbf{W} \otimes A, \bar{B}_i, \bar{C}_i)$, first notice that $(\beta_m \otimes q)^T P_\mathbf{W} = \big ( 0, \cdots, 0,  (\beta_m \otimes q)^T \eta_{\{1, \cdots, t\}}, q^T P_1, 0, \cdots, 0,  q^T P_2\big )$ where $q$ is an (left) eigenvector of $A$. From our construction of $q, P_1, P_2$, we know that $(q^T P_1, ~ q^T P_2)$ is a left eigenvector of $\hat{A}^{(i)}$.
	Thus every left eigenvector of $A_O^{(i)}$ in $\hat{\bar{A}}^{(i)}$ is identified with $\beta_m \otimes q$ for some eigenvector $q$ of $A$. It is straightforward to see that $(\beta_m \otimes q)^T \bar{B_i}$ is nonzero for all $i \in \{1, \cdots, m\}$ and eigenvector $q$. Therefore, every mode of $A_O^{(i)}$ in $\hat{\bar{A}}^{(i)}$ is controllable, and we conclude that the observable part of $(A, C_i)$ is contained in the controllable and observable part of $(\mathbf{W} \otimes A, \bar{B_i}, \bar{C_i})$.
\end{proof}

\subsection{Proof of Lemma \ref{lemma_fixed_mode}} \label{subsection_lemma_fixed_mode}

\begin{proof} [Proof of Lemma \ref{lemma_fixed_mode}]
Let $n_\mathcal{B}$ be the number of columns in $\begin{pmatrix} \mathcal{B}_{i_1} & \cdots & \mathcal{B}_{i_l} \end{pmatrix}$.
	Note that 
	\begin{equation*}
		\begin{split}
			dim \left( Null \begin{pmatrix} \mathcal{A} - \lambda I & \mathcal{B}_{i_1} & \cdots & \mathcal{B}_{i_l}  \end{pmatrix} \right) &= n + n_\mathcal{B} - rank \begin{pmatrix} \mathcal{A} - \lambda I & \mathcal{B}_{i_1} & \cdots & \mathcal{B}_{i_l} \end{pmatrix} \\
			&= n + n_\mathcal{B} - (n - k_\mathcal{B}) = k_\mathcal{B} + n_\mathcal{B}
		\end{split}
	\end{equation*}
	
	Let $\mathbf{0}_{n_\mathcal{B}}$ be a $n_\mathcal{B}$-dimensional zero vector. Since there exist $k_\mathcal{B}$ linearly independent vectors $\left\{ \begin{pmatrix} \eta_j \\ \mathbf{0}_{n_\mathcal{B}} \end{pmatrix} \right\}_{j \in \{1, \cdots, k_\mathcal{B}\}}$ such that $\begin{pmatrix} \eta_j \\ \mathbf{0}_{n_\mathcal{B}} \end{pmatrix} \in Null \begin{pmatrix} \mathcal{A} - \lambda I & \mathcal{B}_{i_1} & \cdots & \mathcal{B}_{i_l}  \end{pmatrix}$ but $\begin{pmatrix} \mathcal{C}_{i_{l+1}} \\ \vdots \\ \mathcal{C}_{i_{m}} \end{pmatrix} \eta_j$ is nonzero for all $j \in \{1, \cdots, k_\mathcal{B}\}$, it holds that
	\begin{equation*}
		\begin{split}
			dim \left( Null \begin{pmatrix} \mathcal{A} - \lambda I & \mathcal{B}_{i_1} & \cdots & \mathcal{B}_{i_l} \\ \mathcal{C}_{i_{l+1}} & 0 & \cdots & 0 \\ \vdots & \vdots & \ddots & \vdots \\ \mathcal{C}_{i_{m}} & 0 & \cdots & 0 \end{pmatrix} \right) \leq k_\mathcal{B} + n_\mathcal{B} - k_\mathcal{B} = n_\mathcal{B}
		\end{split}
	\end{equation*}
	
	Therefore, we obtain
	\begin{equation*}
		rank \begin{pmatrix} \mathcal{A} - \lambda I & \mathcal{B}_{i_1} & \cdots & \mathcal{B}_{i_l} \\ \mathcal{C}_{i_{l+1}} & 0 & \cdots & 0 \\ \vdots & \vdots & \ddots & \vdots \\ \mathcal{C}_{i_{m}} & 0 & \cdots & 0 \end{pmatrix} = n + n_\mathcal{B} - dim \left( Null \begin{pmatrix} \mathcal{A} - \lambda I & \mathcal{B}_{i_1} & \cdots & \mathcal{B}_{i_l} \\ \mathcal{C}_{i_{l+1}} & 0 & \cdots & 0 \\ \vdots & \vdots & \ddots & \vdots \\ \mathcal{C}_{i_{m}} & 0 & \cdots & 0 \end{pmatrix} \right) \geq n
	\end{equation*}
\end{proof}


\bibliographystyle{IEEEtran}
\bibliography{IEEEabrv,spark2012_ACC}

\begin{thebibliography}{10}
\providecommand{\url}[1]{#1}
\csname url@rmstyle\endcsname
\providecommand{\newblock}{\relax}
\providecommand{\bibinfo}[2]{#2}
\providecommand\BIBentrySTDinterwordspacing{\spaceskip=0pt\relax}
\providecommand\BIBentryALTinterwordstretchfactor{4}
\providecommand\BIBentryALTinterwordspacing{\spaceskip=\fontdimen2\font plus
\BIBentryALTinterwordstretchfactor\fontdimen3\font minus
  \fontdimen4\font\relax}
\providecommand\BIBforeignlanguage[2]{{%
\expandafter\ifx\csname l@#1\endcsname\relax
\typeout{** WARNING: IEEEtran.bst: No hyphenation pattern has been}%
\typeout{** loaded for the language `#1'. Using the pattern for}%
\typeout{** the default language instead.}%
\else
\language=\csname l@#1\endcsname
\fi
#2}}

\bibitem{khan2010_cdc}
U.~A. Khan, S.~Kar, A.~Jadbabaie, and J.~M.~F. Moura, ``On connectivity,
  observability, and stability in distributed estimation,'' in \emph{49th IEEE
  Conference on Decision and Control}, Dec. 2010.

\bibitem{matei2010_cdc}
I.~Matei and J.~S. Baras, ``Consensus-based linear distributed filtering,'' in
  \emph{49th IEEE Conference on Decision and Control}, Dec. 2010.

\bibitem{delouille2006}
V.~Delouille, R.~N. Neelamani, and R.~G. Baraniuk, ``Robust distributed
  estimation using the embedded subgraphs algorithm,'' \emph{{IEEE} Trans.
  Signal Processing}, vol.~54, no.~8, Aug. 2006.

\bibitem{smith2007}
R.~S. Smith and F.~Y. Hadaegh, ``Distributed estimation, communication and
  control for deep space formations,'' \emph{IET Control Theory \&
  Applications}, vol.~1, pp. 445--451, Mar. 2007.

\bibitem{khan09}
U.~A. Khan, S.~Kar, and J.~M.~F. Moura, ``Distributed localization and tracking
  with coordinated and uncoordinated motion models,'' in \emph{47th Annual
  Allerton Conference}, Oct. 2009, pp. 202--208.

\bibitem{khan2008_tsp}
U.~A. Khan and J.~M.~F. Moura, ``Distributing the kalman filter for large-scale
  systems,'' \emph{{IEEE} Trans. Signal Processing}, vol.~56, no.~10, Oct.
  2008.

\bibitem{olfati-saber07_ieee_cdc}
R.~Olfati-Saber, ``Distributed kalman filtering for sensor networks,'' in
  \emph{46th IEEE Conference on Decision and Control}, Dec. 2007, pp.
  5492--5498.

\bibitem{carli2007_ieee_cdc}
R.~Carli, A.~Chiuso, L.~Schenato, and S.~Zampieri, ``Distributed kalman
  filtering using consensus strategies,'' in \emph{46th IEEE Conference on
  Decision and Control}, Dec. 2007, pp. 5486--5491.

\bibitem{alriksson2006_mtns}
P.~Alriksson and A.~Rantzer, ``Distributed kalman filtering using weighted
  averaging,'' in \emph{In Proceedings of the 17th International Symposium on
  Mathematical Theory of Networks and Systems}, 2006.

\bibitem{bai2011_acc}
H.~Bai, R.~A. Freeman, and K.~M. Lynch, ``Distributed kalman filtering using
  the internal model average consensus estimator,'' in \emph{2011 American
  Control Conference}, June 29 - July 01 2011, pp. 1500--1505.

\bibitem{farina10}
M.~Farina, G.~Ferrari-Trecate, and R.~Scattolini, ``Distributed moving horizon
  estimation for linear constrained systems,'' \emph{{IEEE} Trans. Automat.
  Contr.}, vol.~55, pp. 2462--2475, Nov. 2010.

\bibitem{sabau2010_allerton}
S.~Sabau and N.~C. Martins, ``On the stabilization of lti decentralized
  configurations under quadratically invariant sparsity constraints,'' in
  \emph{Forty-Eighth Annual Allerton Conference}, Sept. 2010.

\bibitem{wang1973_tac}
S.-H. Wang and E.~J. Davison, ``On the stabilization of decentralized control
  systems,'' \emph{{IEEE} Trans. Automat. Contr.}, vol. AC-18, no.~5, Oct.
  1973.

\bibitem{davison1990_tac}
E.~J. Davison and T.~N. Chang, ``Decentralized stabilization and pole
  assignment for general proper systems,'' \emph{{IEEE} Trans. Automat.
  Contr.}, vol.~35, no.~6, June 1990.

\bibitem{xiao2004_scl}
L.~Xiao and S.~Boyd, ``Fast linear iterations for distributed averaging,''
  \emph{Systems \& Control Letters}, vol.~52, pp. 65--78, 2004.

\bibitem{olfati-saber2007_IEEEproceedings}
R.~Olfati-Saber, J.~A. Fax, and R.~M. Murray, ``Consensus and cooperation in
  networked multi-agent systems,'' \emph{Proceedings of the IEEE}, vol.~95,
  no.~1, Jan. 2007.

\bibitem{horn_topics_in_matrix_analysis}
R.~A. Horn and C.~R. Johnson, \emph{Topics in Matrix Analysis}.\hskip 1em plus
  0.5em minus 0.4em\relax Cambridge University Press, 1994.

\bibitem{axler2004_linear_algebra}
S.~Axler, \emph{Linear Algebra Done Right}.\hskip 1em plus 0.5em minus
  0.4em\relax Springer, 2004.

\bibitem{gong1992_ieee_tac}
Z.~Gong and M.~Aldeen, ``On the characterization of fixed modes in
  decentralized control,'' \emph{{IEEE} Trans. Automat. Contr.}, vol.~37,
  no.~7, 1992.

\bibitem{anderson1981_automatica}
B.~D.~O. Anderson and D.~J. Clements, ``Algebraic characterization of fixed
  modes in decentralized control,'' \emph{Automatica}, vol.~17, no.~5, pp.
  703--712, 1981.

\bibitem{davison1983_automatica}
E.~J. Davison and U.~Ozguner, ``Characterizations of decentralized fixed modes
  for interconnected systems,'' \emph{Automatica}, vol.~19, no.~2, pp.
  169--182, 1983.

\bibitem{brasch1970_tac}
J.~Frederick M.~Brasch and J.~B. Pearson, ``Pole placement using dynamic
  compensators,'' \emph{{IEEE} Trans. Automat. Contr.}, vol. AC-15, no.~1, Feb.
  1970.

\end{thebibliography}

\end{document}